\documentclass[12pt,reqno]{amsart}
\usepackage{amsmath,amsthm,amsfonts,amssymb,amscd,amstext}
\usepackage[latin1]{inputenc}
\usepackage[dvips]{graphicx}
\usepackage{psfrag}
\usepackage{euscript}
\usepackage{a4wide}

\theoremstyle{plain} \newtheorem{maintheorem}{Theorem} 
 \newtheorem{theorem}{Theorem}[section] \newtheorem{proposition}[theorem]{Proposition}
\newtheorem{corollary}[theorem]{Corollary} \newtheorem{lemma}[theorem]{Lemma}  
\theoremstyle{definition} \newtheorem{remark}[theorem]{Remark} \newtheorem{definition}{Definition}

 \newcommand{\RR}{{\mathbb R}}

    \newcommand{\la}{\lambda}

\renewcommand{\epsilon}{\varepsilon}

\newcommand{\diag}{\operatorname{diag}} \newcommand{\trace}{\operatorname{tr}} 
  
  \newcommand{\cof}{\operatorname{cof}}
\newcommand{\adj}{\operatorname{Adj}}

 \newcommand{\sing}{\mathrm{Sing}} \newcommand{\vol}{\operatorname{vol}}

\newcommand{\cC}{\EuScript{C}}    
 \newcommand{\J}{\EuScript{J}}

\begin{document}

\title[ Adapted metrics for codimension one singular hyperbolic flows] { Adapted metrics for codimension one singular hyperbolic flows}


\thanks{L.S. is partially supported by Fapesb-JCB0053/2013, PRODOC-UFBA/2014 and CNPq 2017 postdoctoral schoolarship at Universidade Federal do Rio de Janeiro, V.C. is supported by CAPES}

 \renewcommand{\subjclassname}{\textup{2000} Mathematics5
  Subject Classification}


\author{Luciana Salgado} \address[L.S. \& V.C.]{Universidade Federal da Bahia, Instituto de Matem\'atica\\ Av. Adhemar de Barros, S/N , Ondina,
40170-110 - Salvador-BA-Brazil} \email{lsalgado@ufba.br \& vinicius.coelho@ufba.br}

\author{Vinicius Coelho}

\begin{abstract}
For a partially hyperbolic splitting $T_\Gamma M=E\oplus F$ of
$\Gamma$, a $C^{1}$ vector field $X$ on a $m$-manifold, we obtain singular-hyperbolicity using only the tangent map $DX$ of $X$ and its derivative $DX_{t}$ whether $E$ is one-dimensional subspace. We show the existence of adapted metrics for singular hyperbolic set $\Gamma$ for $C^{1}$  vector fields if $\Gamma$ has a partially hyperbolic splitting $T_\Gamma M=E\oplus F$ where $F$ is volume expanding, $E$ is
uniformly contracted and a one-dimensional subspace.
\end{abstract}


\date{\today}

\maketitle \tableofcontents

\section{Introduction}

Let $M$ be a connected compact finite $m$-dimensional manifold, $m \geq 3$, with or without boundary. We consider a vector field $X$, such that $X$ is inwardly transverse to the boundary $\partial M$, if $\partial M\neq\emptyset$. The flow generated by $X$ is denoted by $X_t$.

A hyperbolic set for a flow $X_t$ on a finite dimensional Riemannian manifold $M$ is a compact invariant set $\Gamma$ with a continuous splitting of the
tangent bundle, $T_\Gamma M= E^s\oplus E^X \oplus E^u$, where $E^X$ is the direction of the vector field, for which the subbundles are invariant under
the derivative $DX_t$ of the flow $X_t$
\begin{align}\label{eq:hyp-splitting}
  DX_t\cdot E^*_x=E^*_{X_t(x)},\quad  x\in \Gamma, \quad t\in\RR,\quad *=s,X,u;
\end{align} and $E^s$ is uniformly contracted by $DX_t$ and $E^u$ is likewise expanded: there are $K,\lambda>0$ so that
\begin{align}\label{eq:Klambda-hyp}
  \|DX_t\mid_{E^s_x}\|\le K e^{-\lambda t},
  \quad
  \|(DX_t \mid_{E^u_x})^{-1}\|\le K e^{-\lambda t},
  \quad x\in \Gamma, \quad t\in\RR.
\end{align} Very strong properties can be deduced from the existence of such hyperbolic structure; see for
instance~\cite{Bo75,BR75,Sh87,KH95,robinson2004}.

An important feature of hyperbolic structures is that it does not depends on the metric on the ambient manifold (see \cite{HPS77}). We recall that a metric is said to be \emph{adapted} to the hyperbolic structure if we can take $K = 1$ in equation (\ref{eq:Klambda-hyp}).

Weaker notions of hyperbolicity (e.g. dominated splitting, partial hyperbolicity, volume hyperbolicity, sectional hyperbolicity, singular hyperbolicity) have been developed to encompass larger classes of systems beyond the uniformly hyperbolic ones; see~\cite{BDV2004} and
specifically~\cite{viana2000i,AraPac2010,ArbSal} for singular hyperbolicity and Lorenz-like attractors.

In the same work \cite{HPS77}, Hirsch, Pugh and Shub asked about adapted metrics for dominated splittings. The positive answer was given by Gourmelon \cite{Goum07} in 2007, where it is given adapted metrics to dominated splittings for both diffeomorphisms and flows, and he also gives an adapted metric for partially hyperbolic splittings as well.

Proving the existence of some hyperbolic structure is, in general, a non-trivial matter, even in its weaker forms.

In \cite{lewow80}, Lewowicz stated that a diffeomorphism on a compact riemannian manifold is Anosov if and only if its derivative admits a nondegenerate Lyapunov quadratic function.

An example of application of the adapted metric from \cite{Goum07} is contained in \cite{arsal2012a}, where the first author jointly with V. Ara\'ujo, following the spirit of Lewowicz's result, construct quadratic forms which characterize partially hyperbolic and singular hyperbolic structures on a trapping region for flows.

In \cite{arsal2015a}, the first author and V. Ara\'ujo provided an alternative way to obtain singular hyperbolicity for three-dimensional flows using the same expression as in Proposition~\ref{pr:J-separated} applied to the infinitesimal generator of the exterior square $\wedge^2 DX_t$ of the cocycle $DX_t$. This infinitesimal generator can be explicitly calculated through the infinitesimal generator $DX$ of the linear multiplicative cocycle $DX_t$
associated to the vector field $X$.

Here, we provide a similar result as above for $m$-dimensional flows if this admits a partially hyperbolic splitting for which one of the invariant subbundles is one-dimensional.

Moreover, we show the existence of adapted metrics for a singular hyperbolic set $\Gamma$ for $C^{1}$  vector fields if $\Gamma$ has a partially hyperbolic splitting $T_\Gamma M=E\oplus F$, where $F$ is volume expanding, $E$ is uniformly contracted and one-dimensional subbundle.

The paper is organized as follow. In Section \ref{sec:stat} we provide definitions and statement of results. In Section \ref{sec:statem-prelim-result} we provide some auxiliary results. Finally, in Section \ref{sec:proof} are given the proofs of our theorems.

\section{Preliminary definitions and statement of results} \label{sec:stat}
\hfill

We now present preliminary definitions and results.

We recall that a \emph{trapping region} $U$ for a flow $X_t$ is an open subset of the manifold $M$ which satisfies: $X_t(U)$ is contained in $U$ for all
$t>0$, and there exists $T>0$ such that $\overline{X_t(U)} $ is contained in the interior of $U$ for all $t>T$. We define $\Gamma(U)=\Gamma_X(U):=
\cap_{t>0}\overline {X_t(U)}$ to be the \emph{maximal
  positive invariant subset in the trapping region $U$}.

A \emph{singularity} for the vector field $X$ is a point $\sigma\in M$ such that $X(\sigma)=\vec0$ or, equivalently, $X_t(\sigma)=\sigma$ for all $t \in
\RR$. The set formed by singularities is the \emph{singular set of
  $X$} denoted $\sing(X)$.  We say that a singularity is
  hyperbolic if the eigenvalues of the derivative
$DX(\sigma)$ of the vector field at the singularity $\sigma$ have nonzero real part.

\begin{definition}\label{eq:domination}
  A \emph{dominated splitting} over a compact invariant set $\Lambda$ of $X$
  is a continuous $DX_t$-invariant splitting $T_{\Lambda}M =
  E \oplus F$ with $E_x \neq \{0\}$, $F_x \neq \{0\}$ for
  every $x \in \Lambda$ and such that there are positive
  constants $K, \lambda$ satisfying
  \begin{align}\label{eq:def-dom-split}
    \|DX_t|_{E_x}\|\cdot\|DX_{-t}|_{F_{X_t(x)}}\|<Ke^{-\la
      t}, \ \textrm{for all} \ x \in \Lambda, \ \textrm{and
      all} \,\,t> 0.
  \end{align}
\end{definition}

A compact invariant set $\Lambda$ is said to be \emph{partially hyperbolic} if it exhibits a dominated splitting $T_{\Lambda}M = E \oplus F$ such that
subbundle $E$ is \emph{uniformly contracted}, i.e., there exists $C>0$ and $\lambda>0$ such that $\|DX_t|_{E_x}\|\leq Ce^{-\lambda t}$ for $t\geq 0$.
In this case $F$ is the \emph{central subbundle} of $\Lambda$.  Or else, we may replace uniform contraction along $E^s$ by uniform expansion along $F$
(the right hand side condition in (\ref{eq:Klambda-hyp}).

We say that a $DX_t$-invariant subbundle $F \subset
  T_{\Lambda}M$ is a \emph{sectionally expanding} subbundle
  if $\dim F_x \geq 2$ is constant for $x\in\Lambda$
  and there are positive constants $C , \lambda$ such that for every $x
    \in \Lambda$ and every two-dimensional linear subspace
    $L_x \subset F_x$ one has
    \begin{align}\label{eq:def-sec-exp}
      \vert \det (DX_t \vert_{L_x})\vert > C e^{\la t},
      \textrm{ for all } t>0.
    \end{align}

\begin{definition}\label{sechypset} \cite[Definition
  2.7]{MeMor06} A \emph{sectional-hyperbolic set} is a
  partially hyperbolic set whose
  central subbundle is sectionally expanding.
\end{definition}

This is a particular case of the so called \emph{singular hyperbolicity} whose definition we recall now.  A $DX_t$-invariant subbundle $F \subset
T_{\Lambda}M$ is said to be a \emph{volume expanding} if in the above condition \ref{eq:def-sec-exp}, we may write
 \begin{align}\label{eq:def-vol-exp}
      \vert \det (DX_t \vert_{F_x})\vert > C e^{\la t},
      \textrm{ for all } t>0.
    \end{align}

\begin{definition}\label{singhypset} \cite[Definition
  1]{MPP99} A \emph{singular hyperbolic set} is a
  partially hyperbolic set whose 
  central subbundle is volume expanding.
\end{definition}

Clearly, in the three-dimensional case, these notions are equivalent.

This is a feature of the Lorenz attractor as proved in \cite{Tu99} and also a notion that extends hyperbolicity for singular flows, because sectional
hyperbolic sets without singularities are hyperbolic; see \cite{MPP04, AraPac2010}.


We assume that coordinates are chosen locally adapted to $\J$ in such a way that $\J(v)=\langle J_x(v),v\rangle, v\in T_xM, x\in U$, and $J_x:T_xM\circlearrowright$ is a self-adjoint linear operator having diagonal matrix with $\pm1$ entries along the diagonal.

We say that a $C^1$ family $\J$ of indefinite and non-degenerate quadratic forms is \emph{compatible} with a continuous splitting $E_\Gamma\oplus
F_\Gamma=E_\Gamma$ of a vector bundle over some compact subset $\Gamma$ if $E_x$ is a $\J$-negative subspace and $F_x$ is a $\J$-positive subspace for
all $x\in\Gamma$.

\begin{proposition}\cite[Proposition 1.3]{arsal2012a}
  \label{pr:J-separated}
  A $\J$-non-negative vector field $X$ on $U$ is strictly
  $\J$-separated if, and only if, there exists a compatible
  family $\J_0$ of forms and there exists a function
  $\delta:U\to\RR$ such that the operator $\tilde J_{0,x}:=
  J_0\cdot DX(x)+DX(x)^*\cdot J_0$ satisfies
  \begin{align*}
    \tilde J_{0,x}-\delta(x)J_0 \quad\text{is positive
      definite}, \quad x\in U,
  \end{align*}
  where $DX(x)^*$ is the adjoint of $DX(x)$ with respect to
  the adapted inner product.
\end{proposition}

\begin{remark}\label{rmk:derivativeJ}
  The expression for $\tilde J_{0,x}$ in terms of $J_0$ and
  the infinitesimal generator of $DX_t$ is, in fact, the time
  derivative of $\J_0$ along the flow direction at the point
  $x$, which we denote $\partial_t J_0$; see item 1 of
  Proposition~\ref{pr:J-separated-tildeJ}. We keep this
  notation in what follows.
\end{remark}


Let $A:G\times\RR\to G$ be a smooth map given by a collection of linear bijections
\begin{align*}
  A_t(x): G_x\to G_{X_t(x)}, \quad x\in\Gamma, t\in\RR,
\end{align*}
where $\Gamma$ is the base space of the finite dimensional vector bundle $G$, satisfying the cocycle property \begin{align*}
  A_0(x)=Id, \quad A_{t+s}(x)=A_t(X_s(x))\circ A_s(x), \quad
  x\in\Gamma, t,s\in\RR,
\end{align*} with $\{X_t\}_{t\in\RR}$ a complete smooth flow over $M\supset\Gamma$.  We note that for each fixed $t>0$ the map $A_t: G\to G, v_x\in G_x \mapsto A_t(x)\cdot v_x\in G_{X_t(x)}$ is an automorphism of the vector bundle $G$.

The natural example of a linear multiplicative cocycle over a smooth flow $X_t$ on a manifold is the derivative cocycle $A_t(x)=DX_t(x)$ on the tangent bundle $G=TM$ of a finite dimensional compact manifold $M$. Another example is given by the exterior power $A_t(x)=\wedge^kDX_t$ of $DX_t$ acting on $G=\wedge^k TM$, the family of all $k$-vectors on the tangent spaces of $M$, for some fixed $1\le k\le\dim G$.

It is well-known that the exterior power of a inner product space has a naturally induced inner product and thus a norm. Thus $G=\wedge^k TM$ has an induced norm from the Riemannian metric of $M$. For more details see e.g. \cite{arnold-l-1998}.

In what follows we assume that the vector bundle $G$ has a smoothly defined inner product in each fiber $G_x$ which induces a corresponding norm $\|\cdot\|_x, x\in\Gamma$.

\begin{definition}\label{def:domcocycle}
  A continuous splitting $G=E\oplus F$ of the vector bundle
  $G$  into a pair of subbundles is \emph{dominated} (with
  respect to the automorphism $A$ over $\Gamma$) if
  \begin{itemize}
  \item the splitting is \emph{invariant}: $A_t(x)\cdot
    E_x=E_{X_t(x)}$ and $A_t(x)\cdot F_x=F_{X_t(x)}$ for all
      $x\in \Gamma$ and $t\in\RR$; and
    \item there are positive
  constants $K, \lambda$ satisfying
  \begin{align}\label{eq:def-dom-split-cocycle}
    \|A_t|_{E_x}\|\cdot\|A_{-t}|_{F_{X_t(x)}}\|<Ke^{-\la
      t}, \ \textrm{for all} \ x \in \Gamma, \ \textrm{and
      all} \,\,t> 0.
  \end{align}
  \end{itemize}
\end{definition}

We say that the splitting $G=E\oplus F$ is \emph{partially
  hyperbolic} if it is dominated and the subbundle $E$ is
uniformly contracted: $\|A_t\mid E_x\|\le Ce^{-\mu t}$ for all $t>0$ and suitable constants $C,\mu>0$.


Let $E_U$ be a finite dimensional vector bundle with inner product $\langle\cdot,\cdot\rangle$ and base given by the trapping region $U\subset M$. Let
$\J:E_U\to\RR$ be a continuous family of quadratic forms $\J_x:E_x\to\RR$ which are non-degenerate and have index $0<q<\dim(E)=n$.  The index $q$ of
$\J$ means that the maximal dimension of subspaces of non-positive vectors is $q$. Using the inner product, we can represent $\J$ by a family of
self-adjoint operators $J_x:E_x\circlearrowleft$ as $\J_x(v)=\langle J_x(v),v\rangle, v\in E_x, x\in U$.

We also assume that $(\J_x)_{x\in U}$ is continuously differentiable along the flow.  The continuity assumption on $\J$ means that for every continuous
section $Z$ of $E_U$ the map $U\ni x\mapsto \J(Z(x))\in\RR$ is continuous. The $C^1$ assumption on $\J$ along the flow means that the map $\RR\ni
t\mapsto \J_{X_t(x)} (Z(X_t(x)))\in \RR$ is continuously differentiable for all $x\in U$ and each $C^1$ section $Z$ of $E_U$.

Using Lagrange diagonalization of a quadratic form, it is easy to see that the choice of basis to diagonalize $\J_y$ depends smoothly on $y$ if the
family $(\J_x)_{x\in U}$ is smooth, for all $y$ close enough to a given $x$. Therefore, choosing a basis for $T_x$ adapted to $\J_x$ at each $x\in U$,
we can assume that locally our forms are given by $\langle J_x(v),v\rangle$ with $J_x$ a diagonal matrix whose entries belong to $\{\pm1\}$,
$J_x^*=J_x$, $J_x^2=I$ and the basis vectors depend as smooth on $x$ as the family of forms $(\J_x)_x$.

We let $\cC_\pm=\{C_\pm(x)\}_{x\in U}$ be the family of positive and negative cones associated to $\J$ \begin{align*}
  C_\pm(x):=\{0\}\cup\{v\in E_x: \pm\J_x(v)>0\},  \quad x\in U,
\end{align*}
and also let $\cC_0=\{C_0(x)\}_{x\in U}$ be the corresponding family of zero vectors $C_0(x)=\J_x^{-1}(\{0\})$ for all $x\in U$.


Let $A:E\times\RR\to E$ be a linear multiplicative cocycle on the vector bundle $E$ over the flow $X_t$. The following definitions are fundamental to state our results.

\begin{definition} \label{def:J-separated} Given a continuous field of non-degenerate quadratic forms $\J$ with constant index on the positively
invariant open subset $U$ for the flow $X_t$, we say that the cocycle $A_t(x)$ over $X_t$ is
\begin{itemize} \item $\J$-\emph{separated} if $A_t(x)(C_+(x))\subset
  C_+(X_t(x))$, for all $t>0$ and $x\in U$ (simple cone invariance);
\item \emph{strictly $\J$-separated} if $A_t(x)(C_+(x)\cup
  C_0(x))\subset C_+(X_t(x))$, for all $t>0$ and $x\in U$
  (strict cone invariance).
  \item $\J$-\emph{monotone} if $\J_{X_t(x)}(DX_t(x)v)\ge \J_x(v)$, for each $v\in
  T_xM\setminus\{0\}$ and $t>0$;
\item \emph{strictly $\J$-monotone} if $\partial_t\big(\J_{X_t(x)}(DX_t(x)v)\big)\mid_{t=0}>0$,
  for all $v\in T_xM\setminus\{0\}$, $t>0$ and $x\in U$;
\item $\J$-\emph{isometry} if $\J_{X_t(x)}(DX_t(x)v) = \J_x(v)$, for each $v\in T_xM$ and $x\in U$.
\end{itemize}
We say that the flow $X_t$ is (strictly) $\J$-\emph{separated} on $U$ if $DX_t(x)$ is (strictly) $\J$-\emph{separated} on $T_UM$. Analogously, the flow of $X$ on $U$ is (strictly) $\J$-\emph{monotone} if  $DX_t(x)$ is (strictly) $\J$-\emph{monotone}.

\end{definition}

\begin{remark}\label{rmk:J-separated-C-}
  If a flow is strictly $\J$-separated, then for $v\in T_xM$
  such that $\J_x(v)\le0$ we have
  $\J_{X_{-t}(x)}(DX_{-t}(v))<0$, for all $t>0$, and $x$ such
  that $X_{-s}(x)\in U$ for every $s\in[-t,0]$. Indeed,
  otherwise $\J_{X_{-t}(x)}(DX_{-t}(v))\ge0$ would imply
  $\J_x(v)=\J_x\big(DX_t(DX_{-t}(v))\big)>0$, contradicting
  the assumption that $v$ was a non-positive vector.

  This means that a flow $X_t$ is strictly
    $\J$-separated if, and only if, its time reversal
    $X_{-t}$ is strictly $(-\J)$-separated. 
\end{remark}

A vector field $X$ is $\J$-\emph{non-negative} on $U$ if $\J(X(x))\ge0$ for all $x\in U$, and $\J$-\emph{non-positive} on $U$ if $\J(X(x))\leq 0$ for
all $x\in U$. When the quadratic form used in the context is clear, we will simply say that $X$ is non-negative or non-positive.

A characterization of dominated splittings, via quadratic forms is given in \cite{arsal2012a} (see also \cite{Wojtk01}) as follow.

\begin{theorem} \label{theo2012} \cite[Theorem  2.13]{arsal2012a}
The cocycle $A_{t}(x)$ is strictly $\J$-separated if, and only if, $E_{U}$ admits a dominated splitting $F_{-} \oplus F_{+}$ with respect to $A_{t}(x)$ on the maximal invariant subset $\Lambda$ of $U$, with constant dimensions $\dim F_{-} = q, \dim F_{+} = p, \dim M = p+q$.
\end{theorem}

This is an algebraic/geometrical way to prove the existence of dominated splittings. In fact, we have an analogous result about partial hyperbolic splittings, as follow.

We say that a compact invariant subset $\Lambda$ is \emph{non-trivial} if
\begin{itemize}
\item either $\Lambda$ does not contain singularities;
\item or $\Lambda$ contains at most finitely many singularities, $\Lambda$ contains some regular orbit and is connected.
    \end{itemize}

\begin{theorem}\cite[Theorem A]{arsal2012a}
  \label{mthm:Jseparated-parthyp}
  A non-trivial compact invariant subset $\Gamma$ is a
  partially hyperbolic set for a flow $X_t$ if, and only if,
  there is a $C^1$ field $\J$ of non-degenerate and
  indefinite quadratic forms with constant index, equal to
  the dimension of the stable subspace of $\Gamma$, such
  that $X_t$ is a non-negative strictly $\J$-separated flow
  on a neighborhood $U$ of $\Gamma$.

  Moreover $E$ is a negative subspace, $F$ a positive
  subspace and the splitting can be made almost orthogonal.
\end{theorem}

Here strict $\J$-separation corresponds to strict cone invariance under the action of $DX_t$ and $\langle,\rangle$ is a Riemannian inner product in the ambient manifold. We recall that the index of a field quadratic forms $\J$ on a set $\Gamma$ is the dimension of the $\J$-negative space at every tangent space $T_xM$ for $x\in U$.  Moreover, we say that the splitting $T_\Gamma M=E\oplus F$ is \emph{almost orthogonal} if, given $\epsilon>0$, there exists a smooth inner product $\langle , \rangle$ on $T_\Gamma M$ so that $|\langle u, v\rangle| < \epsilon$, for all $u \in E, v \in F$, with $\| u \| = 1 = \| v\|$.

We note that the condition stated in Theorem~\ref{mthm:Jseparated-parthyp} allows us to obtain partial hyperbolicity checking a condition at every point of the compact invariant set that depends only on the tangent map $DX$ to the vector field $X$ together with a family $\J$ of quadratic forms without using the flow $X_t$ or its derivative $DX_t$. This is akin to checking the stability of singularity of a vector field using a Lyapunov function.

\subsection{Exterior powers}\label{sec:ext-pow}
\hfill

We note that if $E\oplus F$ is a $DX_t$-invariant splitting of $T_\Gamma M$, with $\{e_1,\dots,e_\ell\}$ a family of basis for $E$ and $\{f_1,\dots,f_h\}$ a family of basis for $F$, then $\widetilde F=\wedge^kF$ generated by $\{f_{i_1}\wedge\dots\wedge f_{i_k}\}_{1\le i_1<\dots<i_k\le h}$ is naturally $\wedge^k DX_t$-invariant by construction. In addition, $\tilde E$ generated by $\{e_{i_1}\wedge\dots\wedge e_{i_k}\}_{1\le i_1<\dots<i_k\le \ell}$ together with all the exterior products of $i$ basis elements of $E$ with $j$ basis elements of $F$, where $i+j=k$ and $i,j\ge1$, is also $\wedge^kDX_t$-invariant and, moreover, $\widetilde E\oplus \widetilde F$ gives a splitting of the $k$th exterior power $\wedge^k T_\Gamma M$ of the subbundle $T_\Gamma M$. Let $T_\Gamma M=E_\Gamma\oplus F_\Gamma$ be a $DX_t$-invariant splitting over the compact $X_t$-invariant subset $\Gamma$ such that $\dim F=k\ge2$.  
  Let $\widetilde F=\wedge^k F$ be the $\wedge^k DX_t$-invariant subspace generated by the vectors of $F$ and $\tilde E$ be the $\wedge^k DX_t$-invariant subspace such that $\widetilde E\oplus\widetilde F$ is a splitting of the $k$th exterior power $\wedge^k T_\Gamma M$ of the subbundle $T_\Gamma M$.

We consider the action of the cocycle $DX_t(x)$ on $k$-vector  that is the $k$-exterior $\wedge^kDX_t$ of the cocycle acting on $\wedge^k T_\Gamma M$.

We denote by $\|\cdot\|$ the standard norm on $k$-vectors induced by the Riemannian norm of $M$, see \cite{arnold-l-1998}.

\begin{remark}
Let $V$ vectorial space of dimension $N$.

\begin{itemize}
\item [($i$)] The dimension of space $\wedge^{r} V$ is $\dim \wedge^{r} V = \begin{pmatrix} N  \\ r
    \end{pmatrix}$.  If $\{e_{1},...,e_{N}\}$ is a basis of $V$, so the set $\{e_{k_{1}}\wedge ... \wedge
    e_{k_{r}}: 1 \leq k_{1} < ... < k_{r} \leq N\}$ is a basis in $\wedge^{r} V$ with $\begin{pmatrix} N  \\
    r  \end{pmatrix}$ elements.
\item [($ii$)] If $V$ has the inner product $\langle, \rangle$, then the bilinear extension of
    $$
    \langle u_{1} \wedge ... \wedge u_{r} , v_{1} \wedge
    ... \wedge v_{r} \rangle :=\det(\langle u_{i},v_{j} \rangle)_{r \times r}
    $$
    defines a inner product in $\wedge^{r}V$. In particular, $||u_{1}
    \wedge ... \wedge u_{r}|| = \sqrt{\det(\langle u_{i},u_{j} \rangle)_{r \times r}}$ is the volume of $r$-dimensional parallelepiped $H$ spanned by $u_{1},...,u_{r}$, we write $\vol(u_{1},...,u_{r}) = \vol(H) = \det(H)= |\det (u_{1},...,u_{r})|$.
    \item [($iii$)]~~ If $A: V \to V$ is a linear operator then the linear extension of
    $\wedge^{r} A(u_{1},...,u_{r}) = A(u_{1}) \wedge ... \wedge A(u_{r})$ defines a linear operator
    $\wedge^{r}A$ on $\wedge^{r}V$.

\item [($iv$)]~~ Let $A: V \to V$, and $\wedge^{r} A: \wedge^{r}V \to \wedge^{r}V$ linear operators with $G$
    spanned by $v_{1},...,v_{s} \in V$. Define $H := A|_{G}$, then $H$ is spanned by $A(v_{1}),...,A(v_{s})$
    . So $|\det A|_{G}| = \vol (A|_{G}) = \vol (H) = \vol (A(v_{1}),...,A(v_{s})) = ||A(v_{1})\wedge ...
    \wedge A(v_{s})|| = ||\wedge^{s}A(v_{1}\wedge ... \wedge v_{s})||$.

\end{itemize}
\end{remark}

When $DX_{t}(u_{i}) = v_{i}(t) = v_{i}$, where G is spanned by $u_{1},...,u_{r} \in T_{\Gamma}M$, and $H$ is spanned by $v_{1},...,v_{r}$, we have $H= DX_{t}(G) = DX_{t}|_{G}$. Thus,

\begin{align*}
|\det(DX_{t}|_{G})| = \vol(DX_{t}(u_{1}),...,DX_{t}(u_{r})) =\\ ||DX_{t}(u_{1}) \wedge... \wedge DX_{t}(u_{r})|| =  ||\wedge^{r} DX_{t}(u_{1} \wedge... \wedge u_{r})||.
\end{align*}

It is natural to consider the linear multiplicative cocyle $\wedge^{k}DX_t$ over the flow $X_t$ of $X$ on $U$, that is, for any $k$ choice, $u_{1},u_{2},...,u_{k}$
of vectors in $T_x M, x\in U$ and $t\in\RR$ such that $X_t(x)\in U$ we set
\begin{align*}
  (\wedge^{k} DX_t)\cdot(u_{1}\wedge u_{2} \wedge ... \wedge u_{k} )=(DX_t\cdot
  u_{1})\wedge (DX_t\cdot
  u_{2})\wedge ... \wedge (DX_t\cdot
  u_{k})
\end{align*}
see \cite[Chapter 3, Section 2.3]{arnold-l-1998} or \cite{Winitzki12} for more details and standard results on exterior algebra and exterior products of linear operator.

In \cite{arsal2015a}, the authors proved the following relation between a dominated splitting and its exterior power.

\begin{theorem}\label{bivectparthyp2} \cite[Theorem A]{arsal2015a}
The splitting $T_{\Gamma}M = E\oplus F$ is dominated for $DX_t$ if, and only if, $\wedge^k T_{\Gamma}M = \widetilde E\oplus \widetilde F$ is a dominated splitting for $\wedge^k DX_t$.
\end{theorem}

Hence, the existence of a dominated splitting $T_\Gamma M=E_\Gamma\oplus F_\Gamma$ over the compact $X_t$-invariant subset $\Gamma$, is equivalent to the bundle  $\wedge^{k}T_\Gamma M$ admits a dominated splitting with respect to $\wedge^{k}DX_t: \wedge^{k}T_\Gamma M \to \wedge^{k}T_\Gamma M $.

As a consequence, they obtain the next characterization of three-dimensional singular sets.

\begin{corollary}\label{corA} \cite[Corollary 1.5]{arsal2015a}  Assume that $M$ has dimension $3$, $E$ is uniformly contracted by $DX_{t}$, and that
$k=2$. Then $E \oplus F$ is a singular-hyperbolic splitting for $DX_{t}$
  if, and only if,  $ \widetilde E \oplus \widetilde F$ is partially hyperbolic splitting for $\wedge^{2}
  DX_t$ such that $\widetilde F$ is uniformly expanded by $\wedge^{2} DX_t$.
\end{corollary}

The main idea in \cite{arsal2015a} was to give a characterization of sectional hyperbolicity following the well known algebraic feature of cross product.

The absolute value of the \emph{cross product} (also called \emph{vector product}) on a $3$-dimensional vector space $V$, denote by $w = u \times v$, provides the length of the vector $w$. It is very useful to calculate the area expansion of the parallelogram generated by $u, v$, under the action of a linear operator.

Following this way, in \cite{arsal2015a}, the first author and V. Araujo proved the result below.

\begin{theorem} \cite[Theorem  $B$]{arsal2015a} \label{theoAS1}
  Suppose that $X$ is $3$-dimensional vector field on $M$
  which is non-negative strictly
  $\J$-separated over a non-trivial subset $\Gamma$, where
  $\J$ has index $1$. 
Then
\begin{enumerate}
\item $\wedge^{2}DX_t$ is strictly $(-\J)$-separated;
\item $\Gamma$ is a singular hyperbolic set if either one of the following
properties is true
\begin{enumerate}
\item $\widetilde\Delta_0^t(x)\xrightarrow[t\to+\infty]{}-\infty$
    for all $x\in\Gamma$.
\item $\tilde \J-2\trace(DX)\J>0$ on $\Gamma$.
\end{enumerate} 
\end{enumerate}
\end{theorem}

Here, we generalized this result to $m$ and $k =m-1$, as follows.

\subsection{Statements of results} \label{sec:statement-result}
\hfill

If $\wedge^{k}DX_t$ is strictly separated with respect to some family $\J$ of quadratic forms, then there exists the function $\delta_k$ as stated in Proposition~\ref{pr:J-separated} with respect to the cocyle $\wedge^{k}DX_t$.  We set $$\widetilde\Delta_a^b(x) := \int_a^b\delta_{k}(X_s(x))\,ds$$ the
area under the function $\delta_{k}:U\to\RR$ given by Proposition~\ref{pr:J-separated} with respect to $\wedge^{k}DX_t$ and its infinitesimal generator.

If $k =m-1$, it is not difficult to see that this function is related to $X$ and $\delta$ as follows: let $\delta:\Gamma\to\RR$ be the function
associated to $\J$ and $DX_t$, as given by Proposition \ref{pr:J-separated}, then $\delta_k=2\trace(DX)-\delta$, where $\trace(DX)$ represents the trace
of the linear operator $DX_x:T_xM\circlearrowleft, x\in M$.

We recall that $\tilde \J = \partial_t \J$ is the time derivative of $\J$ along the flow; see Remark~\ref{rmk:derivativeJ}.

\begin{maintheorem}
  \label{mthm:sec-exp-md}
  Suppose that $X$ is $m$-dimensional vector field on $M$
  which is non-negative strictly
  $\J$-separated over a non-trivial subset $\Gamma$, where
  $\J$ has index $1$. 
Then
\begin{enumerate}
\item $\wedge^{(m-1)}DX_t$ is strictly $(-\J)$-separated; \item $\Gamma$ is a singular hyperbolic set if either one of the following properties is
    true
\begin{enumerate}
\item $\widetilde\Delta_0^t(x)\xrightarrow[t\to+\infty]{}-\infty$
    for all $x\in\Gamma$.
\item $\tilde \J-2\trace(DX)\J>0$ on $\Gamma$.
\end{enumerate}

    \end{enumerate}
    \end{maintheorem}

We  work with exterior products of codimension one. See \cite{elon} for more details.

This result provides useful sufficient conditions for a $m$-dimensional vector field to be singular hyperbolic if $k=m-1$, using only one family of
quadratic forms $\J$ and its space derivative $DX$, avoiding the need to check cone invariance and contraction/expansion conditions for the flow $X_t$
generated by $X$ on a neighborhood of $\Gamma$.

\begin{definition} \label{adapmetric3d} We say a Riemannian metric $\langle\cdot,\cdot\rangle$ \emph{adapted to a singular hyperbolic splitting}
$T\Gamma = E \oplus F$ if it induces a norm $|\cdot|$ such that there exists $\lambda>0$ satisfying for all $x\in \Gamma$ and $t>0$ simultaneously
    \begin{align*}
    |DX_t\mid_{E_x}|\cdot\big|(DX_t\mid_{F_x})^{-1}|
    \le e^{-\lambda t}, \ 
     |DX_t\mid_{E_x}|\le e^{-\lambda t}
\quad\text{and}\quad \vert \det (DX_t \mid_{F_x})\vert \ge e^{\lambda t}. \end{align*} We call it \emph{singular adapted metric}, for simplicity.
\end{definition}

In \cite{arsal2015a}, the first author and V. Araujo proved the following.

\begin{theorem} \cite[Theorem  $C$]{arsal2015a} \label{theoAS2}
  Let $\Gamma$ be a singular-hyperbolic
  set for a $C^1$ three-dimensional vector field $X$.
  Then $\Gamma$ admits a singular adapted metric.
\end{theorem}

Here, we generalize this result for any codimension one singular hyperbolic flow in higher dimensional manifolds. Consider a partially hyperbolic
splitting $T_\Gamma M=E\oplus F$ where $E$ is uniformly contracted and $F$ is volume expanding. We show that  for $C^1$ flows having a
singular-hyperbolic set $\Gamma$ such that  $E$ is one-dimensional subspace there exists a metric adapted to the partial hyperbolicity and the area
expansion, as follows.

\begin{maintheorem}
  \label{mthm1}
	Let $\Gamma$ be a singular-hyperbolic
  set of codimension one for a $C^1$ $m$-dimensional vector field $X$.
  partially hyperbolic splitting satisfying (\ref{eq:domination}) together with  uniform contraction along $E$
  and volume expanding  along $F$ such that $\dim E = 1$.
  Then $\Gamma$ admits a singular adapted metric.
\end{maintheorem}


\begin{remark} Using the Theorem \ref{theoAS1}, the first author and V. Araujo provided in \cite{arsal2015a} a proof of the Theorem \ref{theoAS2}.
Analogously, we can show the Theorem \ref{mthm1} using the Theorem \ref{mthm:sec-exp-md}. Here, we'll present an alternative proof for Theorem
\ref{mthm1}  that is  independent  of the Theorem \ref{mthm:sec-exp-md}. \end{remark}

We briefly present these results in what follows with the relevant definitions.

\section{Auxiliary results}\label{sec:statem-prelim-result}

\subsection{Properties of $\J$-separated linear multiplicative cocycles} \label{sec:propert-j-separat}
\hfill

We present some useful properties about $\J$-separated linear cocycles whose proofs can be found in \cite{arsal2012a}.

Let $A_t(x)$ be a linear multiplicative cocycle over $X_t$. We define the infinitesimal generator of $A_t(x)$ by
\begin{align}\label{eq:infinitesimal-gen}
  D(x):=\lim_{t\to0}\frac{A_t(x)-Id}t.
\end{align}

The following is the basis of our arguments leading to Theorem~\ref{mthm:Jseparated-parthyp}.

The area under the function $\delta$ provided by Proposition~\ref{pr:J-separated-tildeJ} allows us to detect different dominated splittings with respect
to linear multiplicative cocycles on vector bundles (Proposition~\ref{pr:char-dom-split}). For this, define the function
    \begin{align}\label{eq:delta-area}
     \Delta_a^b(x):=\int_a^b\delta(X_s(x))\,ds, \quad
     x\in\Gamma, a,b\in\RR.
    \end{align}

\begin{proposition}\cite[Proposition 2.7]{arsal2012a}
  \label{pr:J-separated-tildeJ}
  Let $A_t(x)$ be a cocycle over $X_t$ defined on an open
  subset $U$ and $D(x)$ its infinitesimal generator. Then
  \begin{enumerate}
  \item $\tilde\J(v)=\partial_t \J(A_t(x)v) =
    \langle \tilde J_{X_t(x)} A_t(x)v,A_t(x)v\rangle$ for
    all $v\in E_x$ and $x\in U$, where
    \begin{align}\label{eq:J-separated-tildeJ}
      \tilde J_x:= J\cdot D(x) + D(x)^* \cdot J
    \end{align}
    and $D(x)^*$ denotes the adjoint of the linear map
    $D(x):E_x\to E_x$ with respect to the adapted inner
    product at $x$;
  \item the cocycle $A_t(x)$ is $\J$-separated if, and only
    if, there exists a neighborhood $V$ of $\Lambda$,
    $V\subset U$ and a function $\delta:V\to\RR$ such that
    \begin{align}\label{eq:J-ge}
      \tilde \J_x\ge\delta(x)\J_x
      \quad\text{for all}\quad x\in V.
    \end{align}
    In particular we get $\partial_t\log|\J(A_t(x)v)|\ge
    \delta(X_t(x))$, $v\in E_x, x\in V, t\ge0$;
  \item if the inequalities in the previous item are strict,
    then the cocycle $A_t(x)$ is strictly
    $\J$-separated. Reciprocally, if $A_t(x)$ is strictly
    $\J$-separated, then there exists a compatible family
    $\J_0$ of forms on $V$ satisfying the strict
    inequalities of item (2).
  \item For a $\J$-separated cocycle $A_t(x)$, we have
    $\frac{|\J(A_{t_2}(x)v)|}{|\J(A_{t_1}(x)v)|}\ge \exp
    \Delta_{t_1}^{t_2}(x)$ for all $v\in E_x$ and reals
    $t_1<t_2$ so that $\J(A_t(x)v)\neq0$ for all $t_1\le
    t\le t_2$, where $\Delta_{t_1}^{t_2}(x)$ was defined
    in~(\ref{eq:delta-area}).
\item we can bound $\delta$ at every $x\in\Gamma$ by
    $\inf_{v\in C_+(x)}\frac{\tilde\J(v)}{\J(v)}
    \le\delta(x)\le
    \sup_{v\in C_-(x)}\frac{\tilde\J(v)}{\J(v)}.$
  \end{enumerate}
\end{proposition}

\begin{remark}
  \label{rmk:strictly-J-separated}
  We stress that the necessary and sufficient condition in
  items (2-3) of Proposition~\ref{pr:J-separated-tildeJ},
  for (strict) $\J$-separation, shows that a cocycle
  $A_t(x)$ is (strictly) $\J$-separated if, and only if, its
  inverse $A_{-t}(x)$ is (strictly) $(-\J)$-separated.
\end{remark}

\begin{remark}
  \label{rmk:Jexp-ineq}
  Item (2) above of Proposition~\ref{pr:J-separated-tildeJ}
  shows that $\delta$ is a measure of the ``minimal
  instantaneous expansion rate'' of $|\J\circ A_t(x)|$.
\end{remark}

\begin{proposition}\cite[Theorem 2.23]{arsal2012a}
  \label{pr:char-dom-split}
  Let $\Gamma$ be a compact invariant set for $X_t$
  admitting a dominated splitting $E_\Gamma= F_-\oplus F_+$
  for $A_t(x)$, a linear multiplicative cocycle over
  $\Gamma$ with values in $E$. Let $\J$ be a $C^1$ family
    of indefinite quadratic forms such that $A_t(x)$ is
    strictly $\J$-separated. Then
  \begin{enumerate}
  \item $F_-\oplus F_+$ is partially hyperbolic with 
    $F_+$ uniformly expanding
    if 
    $\Delta_0^t(x)\xrightarrow[t\to+\infty]{} +\infty$
    for all $x\in\Gamma$.
  \item $F_-\oplus F_+$ is partially hyperbolic with $F_-$
    uniformly contracting 
    if 
    $\Delta_0^t(x)\xrightarrow[t\to+\infty]{}-\infty$
    for all $x\in\Gamma$.
  \item $F_-\oplus F_+$ is uniformly hyperbolic
    if, and only if, there exists a compatible family $\J_0$
    of quadratic forms in a neighborhood of $\Gamma$ such
    that $\J_0'(v)>0$ for all $v\in E_x$ and all $x\in\Gamma$.
  \end{enumerate}
\end{proposition}

For the proof and more details about the Proposition \ref{pr:char-dom-split}, see \cite{arsal2012a}.

Above we write $\tilde\J(v)= <\tilde J_x v, v>$, where $\tilde J_x$ is given in Proposition~\ref{pr:J-separated}, that is, $\tilde\J(v)$ is the time
derivative of $\J$ under the action of the flow.

We use Proposition~\ref{pr:char-dom-split} to obtain sufficient conditions for a flow $X_t$ on $m$-manifold to have a $\wedge^{m-1} DX_t$-invariant
one-dimensional uniformly expanding direction orthogonal to the $(m-1)$-dimensional center-unstable bundle.


First, we present  some properties about exterior products and the main lemma to prove the theorem \ref{mthm:sec-exp-md}.

Let $V$ a $m$-dimensional vector  space, we denote $V$ by $V^{m}$,  consider $\wedge^{k} V^{m}$ where $2 \leq k \leq m$. Let
$\mathcal{B}=\{e_{1},...,e_{m}\}$ a basis of $V^{m}$. So $\{e_{j_{1}}\wedge ... \wedge e_{j_{k}}: 1 \leq j_{1} < ... < j_{k} \leq m \}$ is a basis of
$\wedge^{k} V^{m}$, and $J := \{ (j_{1},..., j_{k}) \in \mathbb{N}^{k}: 1 \leq j_{1} < ... < j_{k} \leq m \}$. Let $l = \begin{pmatrix} m \\ k \\
\end{pmatrix}$, so we have $l$ combination of $k$ vectors in $\{e_{1},...,e_{m}\}$, and  $|J| = l$.


Take $u_{1},u_{2},...,u_{k} \in V^{m}$ where $u_{j} = (u^{1}_{j},u^{2}_{j},...,u^{m}_{j})_{\mathcal{B}}$ for all $j \in \{1,...,k\}$. Define

\begin{equation} \mathcal{C} := \begin{pmatrix} 	u_{1}^{1} & ... & u_{k}^{1}   \\ ... & ... & ...  \\ u_{1}^{m} & ... & u_{k}^{m}    \\
\end{pmatrix}_{m \times k}. \end{equation}

For $(j_{1},...,j_{k}) \in J$, consider

\begin{equation}\label{submatrix}
\mathcal{C}^{j_{1},...,j_{k}}:= \begin{pmatrix} 	
u_{1}^{j_{1}} & ... & u_{k}^{j_{1}}   \\ ... & ... & ...  \\
u_{1}^{j_{k}} & ... & u_{k}^{j_{k}}    \\
\end{pmatrix}_{k \times k}
\end{equation}

The following result holds

\begin{equation}\label{eq1}
 u_{1} \wedge ...\wedge u_{k}= \sum_{(j_{1},..., j_{k}) \in J} \det (\mathcal{C}^{j_{1},...,j_{k}})
 (e_{j_{1}}\wedge ... \wedge e_{j_{k}}).
\end{equation}

Let $A: V^{m} \to V^{m}$ a linear operator with matrix in basis $\mathcal{B}$ given by

\begin{equation}
\begin{pmatrix} 	
a_{11} & a_{12} & ... & a_{1m}  \\ ... & ... & ... & ... \\ a_{m1} & a_{m2} & ... & a_{mm}  \\
\end{pmatrix}_{(m\times m)}.
\end{equation}

We will denote this matrix by $A$ too.


Consider $\wedge^{k} A: \wedge^{k}  V^{m}  \to \wedge^{k}  V^{m}$, note that $A(u_{1}) \wedge ... \wedge A(u_{k}) = \wedge^{k} A (u_{1} \wedge ...\wedge u_{k})$, by (\ref{eq1}) and the linearity of  $\wedge^{k} A$, we have that

\begin{equation}
A(u_{1}) \wedge ... \wedge A(u_{k}) = \sum_{(j_{1},..., j_{k}) \in J} \det (\mathcal{C}^{j_{1},...,j_{k}})\wedge^{k} A (e_{j_{1}}\wedge ... \wedge e_{j_{k}})
\end{equation}

Define $A_{j}:=A(e_{j})$, so $A_{j}$ is the $j$-th column of $A$, i.e., $A(e_{j})=A_{j} = (a_{1j},...,a_{mj})^{T}$, so $A(e_{j}) = [a_{ij}]_{m \times
1}$. Let $A_{j_{1} ...j_{k}}:=(A_{j_{1}}$ $...$ $A_{j_{k}})_{m\times k} $ where $(j_{1},...,j_{k}) \in J$.
For each $(i_{1}...i_{k}),(j_{1}...j_{k}) \in J$ consider

\begin{equation}
A^{i_{1}...i_{k}}_{j_{1}...j_{k}}:=
\begin{pmatrix} 	
a_{i_{1}j_{1}} & ... & a_{i_{1}j_{k}}   \\ ... & ... & ...  \\ a_{i_{k}j_{1}} &
... & a_{i_{k}j_{k}}   \\
\end{pmatrix}_{k \times k}
\end{equation}

Using that $\wedge^{k} A (e_{j_{1}}\wedge ... \wedge e_{j_{k}}) = A(e_{j_{1}})\wedge ...\wedge A(e_{j_{k}})$ with matrix
$$
A_{j_{1} ...j_{k}}:=(A_{j_{1}}...A_{j_{k}})_{m\times k},
$$
by (\ref{eq1}) we obtain that

\begin{equation}
A(e_{j_{1}})\wedge ... \wedge A(e_{j_{k}}) = \sum_{(i_{1},..., i_{k}) \in J} \det (A^{i_{1}...i_{k}}_{j_{1}...j_{k}}) (e_{i_{1}}\wedge ... \wedge e_{i_{k}}).
\end{equation}

\begin{lemma} \label{lemmatheoA}
Let $V$ a $m$-dimensional manifold and $A:T_xV\to T_yV$ a linear operator  then $\wedge^{(m-1)} A=\det(A)\cdot
(A^{-1})^*$.
\end{lemma}

\begin{proof}

Consider $k=m-1$.

We use the following identification between $\wedge^{(m-1)}T_{x}M$ and $T_{x}M$. For each $(j_{1},...,j_{m-1}) \in J$, we identify $e_{j_{1}} \wedge ... \wedge e_{j_{k}}$ in $\wedge^{(m-1)}T_{x}M$ by $\delta_{p}e_{p}$ in $T_{x}M$, where $p \notin \{j_{1},...,j_{m-1}\}$, $\delta_{p} = 1$ if $p$ is odd, and $\delta_{p} = -1$ if $p$ is even.

We must show that for each  $(j_{1},...,j_{m-1}) \in J$ the exterior product $\wedge^{(m-1)} A (e_{j_{1}} \wedge ... \wedge e_{j_{k}})$ corresponds to the $\det(A)\cdot (A^{-1})^* (\delta_{p}e_{p})$, where $\delta_{p}e_{p}$ is given as above.

Define $S:= \det (A)\cdot (A^{-1})^*$, using that $A^{-1}=\frac{1}{\det(A)} \adj (A)$, we obtain that $S= \cof(A)$ where $\cof(A) = [(-1)^{i+j}M_{ij}]_{m \times m}$ and $M_{ij}$ is the determinant of the submatrix formed by deleting the $i$-th row and $j$-th column. We have that
$M_{ij} = \det (A^{r_{1}...r_{k}}_{s_{1}...s_{k}})$ where $i \notin {r_{1},...,r_{k}}$ and $j \notin {s_{1},...,s_{k}}$.

Note that
$$
\cof(A) (\delta_{p}e_{p})=\delta_{p} \cof(A)(e_{p}) = \delta_{p} ((-1)^{1+p}M_{1p},(-1)^{2+p}M_{2p},...,(-1)^{m+p}M_{mp})_{\mathcal{B}}.
$$

In case $p$ is odd, $\delta_{p} = 1$ and $\cof(A) (\delta_{p}e_{p}) = (M_{1p},-M_{2p},...,(-1)^{m+p}M_{mp})_{\mathcal{B}}$.

We obtain that
\begin{align*}
\cof(A) (\delta_{p}e_{p}) = M_{1p}e_{1}+M_{2p}(-e_{2})+...+M_{mp}(-1)^{m+p} =\\
M_{1p}(e_{1}\delta_{1})+M_{2p}(e_{2}\delta_{2})+...+M_{mp}(e_{mp}\delta_{mp}).
  \end{align*}

Using that
$$
A(e_{j_{1}})\wedge ... \wedge A(e_{j_{k}}) = \sum_{(i_{1},..., i_{k}) \in J} \det (A^{i_{1}...i_{k}}_{j_{1}...j_{k}}) (e_{i_{1}}\wedge ... \wedge e_{i_{k}})
$$
and
$M_{ij} = \det (A^{r_{1}...r_{k}}_{s_{1}...s_{k}})$  where $i \notin {r_{1},...,r_{k}}$  and $j \notin {s_{1},...,s_{k}}$, we have that $\cof(A) (\delta_{p}e_{p}) \cong A(e_{j_{1}})\wedge ... \wedge A(e_{j_{k}})$. 

This concludes the proof.
\end{proof}

We generalize the corollary \ref{corA} to arbitrary $n$ and $k$. 

We just need to prove the following result:

\begin{lemma}\label{bivectparthyp1}
The subbundle $F_{\Gamma}$ is volume expanding by $DX_t$ if, and only if,  $\widetilde F$ is uniformly expanded by $\wedge^{k} DX_t$.

In particular,  $E \oplus F$  is a singular hyperbolic splitting, where $F$ is volume expanding for $DX_{t}$ if, and only if, $ \widetilde E \oplus \widetilde F$ is partially hyperbolic splitting for $\wedge^{k} DX_t$ such that $\widetilde F$ is uniformly expanded by $\wedge^{k} DX_t$.
\end{lemma}

\begin{proof}

We consider the action of the cocycle $DX_t(x)$ on $k$-vector  that is the $k$-exterior $\wedge^kDX_t$ of the cocycle acting on $\wedge^kT_\Gamma M$.

Denote by $\|\cdot\|$ the standard norm on $k$-vectors induced by the Riemannian norm of $M$; see, e.g. \cite{arnold-l-1998}. We write $m=\dim M$.

Suppose that $T_\Gamma M$ admits a splitting $E_\Gamma\oplus F_\Gamma$ with $\dim E_\Gamma=m-k$
  and $\dim F_\Gamma=k$.
 		
We note that if $E\oplus F$ is a $DX_t$-invariant splitting of $T_\Gamma M$, with $\{e_1,\dots,e_l\}$ a family of basis for $E$ and $\{f_1,\dots,f_k\}$
a family of basis for $F$, then $\widetilde F=\wedge^kF$ generated by $\{f_{i_1}\wedge\dots\wedge f_{i_k}\}_{1\le i_1<\dots<i_k\le k}$ is naturally
$\wedge^kDX_t$-invariant by construction. Then, the dimension of $\widetilde F$ is one with basis given by the vector $f_1\wedge ... \wedge f_k$.

Assume that $F_{\Gamma}$ is volume expanding by $DX_t$.  We must show that there exist $C$ and $\lambda > 0$ such that $|\wedge^{k}DX_{t}|_{P}| \geq
Ce^{\lambda t}$, for all $t>0$, where $P$ is spanned by $f_{1} \wedge ... \wedge f_{k}$.

Note that
$$
||\wedge^{k}DX_{t}|_{P}|| =||\wedge^{k}DX_{t}(f_{1} \wedge ... \wedge f_{k})||  =||DX_{t}(f_{1}) \wedge ... \wedge DX_{t}(f_{k})||.
$$

But $f_{1},...,f_{k}$ is a basis for $F$, by hypothesis there exist constants $C$ and $\lambda > 0$ such that $|\det(DX_{t}|_{F})| \geq C.e^{\lambda t}$
for all $t>0$. So,

$|\det(DX_{t}|_{F})| = \vol(DX_{t}(f_{1}),...,DX_{t}(f_{k})) =||DX_{t}(f_{1})\wedge ... \wedge DX_{t}(f_{k})||$.

The reciprocal statement is straightforward.

Given a basis $\{f_{1},...,f_{k}\}$ of $F$, we have that
  \begin{align*}
  |\det (DX_{t}|_{F})|= \\ \vol(DX_{t}(f_{1}),...,DX_{t}(f_{k})) = ||DX_{t}(f_{1})\wedge ...\wedge
  DX_{t}(f_{k})|| =\\ ||\wedge^{k}DX_{t}(f_{1}\wedge ...\wedge f_{k})||=||\wedge^{k}DX_{t}|_{P}||
    \end{align*}
    where $P$ is spanned by $f_{1} \wedge ... \wedge f_{k}$.
	
However, by hypothesis, there exist $C$ and $\lambda > 0$ such that $||\wedge^{k}DX_{t}|_{P}|| \geq Ce^{\lambda t}$, for all $t>0$. 
\end{proof}

\begin{corollary}\label{bivectparthyp} Assume that $E$ is uniformly contracted by $DX_{t}$. 	$E \oplus F$ is a singular-hyperbolic splitting for
$DX_{t}$
  if, and only if,  $ \widetilde E \oplus \widetilde F$ is partially hyperbolic splitting for $\wedge^{k}
  DX_t$ such that $\widetilde F$ is uniformly expanded by $\wedge^{k} DX_t$.
\end{corollary}

Let $M$ Riemannian manifold $m$-dimensional with $\langle \cdot,\cdot \rangle$ inner product in $T_{\Gamma}M$, and $\langle \cdot,\cdot \rangle_{*}$ the
inner product in $\wedge^{k} T_{\Gamma}M$ induced by $\langle \cdot,\cdot \rangle$ where $\wedge^{k} T_{\Gamma}M = \bigcup_{x \in \Gamma} \wedge^{k}
T_{x}M$. So for $x \in \Gamma$, we have that $\langle \cdot,\cdot \rangle$ is acting on $T_{x}M$,  and  $\langle \cdot,\cdot \rangle_{*}$ is acting on
$\wedge^{k} T_{x}M$. 	

\begin{lemma} \label{lemma1} 
Let $M$ be a riemannian $m$-dimensional manifold. Then, for all $[\cdot,\cdot]_{*}$ inner product in  $\wedge^{(m-1)} T_{\Gamma}M$ there exists a inner product $[\cdot,\cdot]$ on $T_{\Gamma}M$ such that $[\cdot,\cdot]_{*}$ is induced by  $[\cdot,\cdot]$. 
\end{lemma}

\begin{proof} 
Let $M$ be a riemannian manifold $m$-dimensional with $\langle \cdot,\cdot \rangle$ an inner product in $T_{\Gamma}M$, and $\langle \cdot,\cdot \rangle_{*}$ the inner product in $\wedge^{(m-1)} T_{\Gamma}M$ induced by $\langle \cdot,\cdot \rangle$.

Take $[\cdot,\cdot]_{**}$ an arbitrary inner product in  $\wedge^{(m-1)} T_{\Gamma}M$. Using that $[\cdot,\cdot]_{**}$ and $\langle \cdot, \cdot \rangle_{*}$ are inner products in $\wedge^{(m-1)} T_{\Gamma}M$ there exists $J:  \wedge^{(m-1)} T_{\Gamma}M \to \wedge^{(m-1)} T_{\Gamma}M$ isomorphism linear such that $[u,v]_{**} = \langle J(u), J(v) \rangle_{*}$.

Define $\varphi : GL(T_{\Gamma}M) \to GL(\wedge^{(m-1)} T_{\Gamma}M)$ given by $A \mapsto \wedge^{(m-1)} A$. 

Note that $\varphi$ is an injective linear homomorphism, and due to the dimensions of the spaces,  $\varphi$ is a linear isomorphism.

Hence, there exists $A \in   GL(T_{\Gamma}M)$ such that  $\wedge^{(m-1)} A = J$. 

Consider $[x,y] :=\langle A(x), A(y) \rangle$ for $x, y \in T_{\Gamma}M$, then $[u,v]_{*} = \det([u_{i},v_{j}])_{(m-1)\times (m-1)}$, where $u = u_{1} \wedge ... \wedge u_{(m-1)}$ and $v = v_{1} \wedge ... \wedge v_{(m-1)}$. 

We have that 
$$
[u,v]_{*} = \det(\langle A(u_{i}),A(v_{j})\rangle)_{(m-1)\times (m-1)}.
$$

On the other hand, 
$$
[u,v]_{**} = \langle \wedge^{(m-1)} A (u), \wedge^{(m-1)} A (v) \rangle_{*}= \det(\langle A(u_{i}),A(v_{j})\rangle)_{(m-1)\times (m-1)}.
$$

Therefore,  $[\cdot,\cdot]_{*} = [\cdot,\cdot]_{**}$, and we are done. 
\end{proof}


\section{Proofs of main results}\label{sec:proof}

\subsection{Proof of Theorem ~\ref{mthm:sec-exp-md}}

\begin{proof}
Consider $M$ is a $m$-manifold and $\Gamma$ is a compact $X_t$-invariant subset having a singular-hyperbolic splitting $T_\Gamma M=E_\Gamma\oplus
F_\Gamma$.  By Theorem~\ref{bivectparthyp2} we have a $\wedge^{(m-1)}DX_t$-invariant partial hyperbolic splitting $\wedge^{(m-1)} T_\Gamma M=\widetilde
E\oplus \widetilde F$ with $\dim\widetilde F=1$ and $\widetilde F$ uniformly expanded. Following the proof of Theorem~\ref{bivectparthyp2}, if we write
$e$ for a unit vector in $E_x$ and $\{u_{1},u_{2},...,u_{m-1}\}$ an orthonormal base for $F_x$, $x\in\Gamma$, then $\widetilde E_x$ is a
$(m-1)$-dimensional vector space spanned by  set $\{ e \wedge u_{i_{1}}\wedge u_{i_{2}} \wedge ... \wedge u_{i_{m-2}}\}$ with $i_{1},...,i_{m-2} \in
\{1,...,m-1\}$.

From Theorem~\ref{mthm:Jseparated-parthyp} and the existence of adapted metrics (see e.g.  \cite{Goum07}), there exists a field $\J$ of quadratic forms
so that $X$ is $\J$-non-negative, $DX_t$ is strictly $\J$-separated on a neighborhood $U$ of $\Gamma$, $E_\Gamma$ is a negative subbundle, $F_\Gamma$ is
a positive subbundle and these subspaces are almost orthogonal. In other words, there exists a function $\delta:\Gamma\to\RR$ such that
$\tilde\J_x-\delta(x)\J_x>0, x\in\Gamma$ and we can locally write $\J(v)=\langle J(v),v\rangle$ where $J=\diag\{-1,1,...,1\}$ with respect to the basis
$\{e,u_{1},...,u_{m-1}\}$ and $\langle\cdot,\cdot\rangle$ is the adapted inner product; see \cite{arsal2012a}.

By lemma \ref{lemmatheoA}, $\wedge^{(m-1)} A=\det(A)\cdot (A^{-1})^*$ with respect to the adapted inner product which trivializes $\J$, for any linear
transformation $A:T_xM\to T_yM$. Hence $\wedge^{m-1}DX_t(x)=\det (DX_t(x))\cdot (DX_{-t}\circ X_t)^*$ and a straightforward calculation shows that the
infinitesimal generator $D^{(m-1)}(x)$ of $\wedge^{(m-1)}DX_t$ equals $\trace(DX(x))\cdot Id - DX(x)^*$.

Therefore, using the identification between $\wedge^{(m-1)} T_xM$ and $T_xM$ through the adapted inner product and
Proposition~\ref{pr:J-separated-tildeJ} \begin{align}
  \hat\J_x=\partial_t(-\J)(\wedge^{(m-1)}DX_t\cdot
  v)\mid_{t=0} &= \langle-(J\cdot D^{(m-1)}(x)+D^{(m-1)}(x)^*\cdot
  J)v,v\rangle\nonumber
  \\
  &= \langle [(\J\cdot
  DX(x)+DX(x)^*\cdot\J)-2\trace(DX(x))\J]v,v\rangle\nonumber
  \\
  &= (\tilde\J-2\trace(DX(x))\J)(v).\label{eq:square-generator}
\end{align} To obtain strict $(-\J)$-separation of $\wedge^{(m-1)}DX_t$ we search a function $\delta_{(m-1)}:\Gamma\to\RR$  so that \begin{align*}
  (\tilde\J-2\trace(DX)\J) - \delta_{(m-1)}(-\J)>0
  \quad\text{or}\quad
  \tilde\J-(2\trace(DX)-\delta_{(m-1)})\J>0.
\end{align*} Hence it is enough to make $\delta_{(m-1)}=2\trace(DX)-\delta$. This shows that in our setting $\wedge^{(m-1)}DX_t$ is always strictly
$(-\J)$-separated.

Finally, according to Theorem~\ref{pr:char-dom-split}, to obtain the partial hyperbolic splitting of $\wedge^{(m-1)}DX_t$ which ensures
singular-hyperbolicity, it is 
sufficient that either $\widetilde\Delta_a^b(x)=\int_a^b\delta_{(m-1)}(X_s(x))\,ds$ satisfies item (1) of Proposition~\ref{pr:char-dom-split} or
$\hat\J_x$ is positive definite, for all $x\in\Gamma$. This amounts precisely to the 
sufficient condition in the statement of Theorem~\ref{mthm:sec-exp-md} and we are done.
\end{proof}

\subsection{Proof of Theorem ~\ref{mthm1}} \begin{proof} Let singular-hyperbolic  set $\Gamma$ for vector field with a splitting $E\oplus F$ where $E$
is uniformly contracted and $F$ is volume expanding.

Suppose that $T_\Gamma M$ admits a splitting $E_\Gamma\oplus F_\Gamma$ with $\dim E_\Gamma=1$
  and $\dim F_\Gamma=k=m-1$.
 		
We note that if $E\oplus F$ is a $DX_t$-invariant splitting of $T_\Gamma M$, with $\{e_1\}$ a basis for $E$ and $\{f_1,\dots,f_k\}$ a family of basis
for $F$, then $\widetilde F=\wedge^kF$ generated by $\{f_{i_1}\wedge\dots\wedge f_{i_k}\}_{1\le i_1<\dots<i_k\le k}$ is naturally
$\wedge^kDX_t$-invariant by construction. Then, the dimension of $\widetilde F$ is one with basis given by the vector $f_1\wedge ... \wedge f_k$.

 By corollary \ref{bivectparthyp}, we have a partially hyperbolic splitting $\widetilde E \oplus \widetilde F$
 for $\wedge^{k} DX_t$ such that $\widetilde F$ is uniformly expanded by $\wedge^{k} DX_t$. Hence, from
 \cite[Theorem 1]{Goum07} , there
exists an adapted inner product $\langle\cdot,\cdot\rangle_{*}$ for $\wedge^{k} DX_t$. There exists   $\lambda>0$ satisfying for all $x\in \Gamma$ and
$t>0$ such that  $||\wedge^{k} DX_t\mid_{ \widetilde F_x}|| \geq e^{\lambda t}$ for all $t>0$.

By Lemma \ref{lemma1}, $\langle\cdot,\cdot\rangle_{*}$ is  induced by an inner product $\langle \cdot , \cdot \rangle$ in $T_{\Gamma}M$. So, we have a
partially hyperbolic splitting $\widetilde E \oplus \widetilde F$  for $\wedge^{k} DX_t$ such that $\widetilde F$ is uniformly expanded by $\wedge^{k}
DX_t$. By Theorem \ref{bivectparthyp2}, we have that $E\oplus F$ is a dominated splitting for $DX_t$. From Theorem \ref{theo2012}, there exists $C^{1}$
field of quadratic $\J$ such that  $DX_{t}$ is strictly $\J$-separated.

But  $DX_{t}$ is strictly $\J$-separated, this ensures, in particular, that the norm

$|w|=\xi\sqrt{\J(w_E)^2+\J(w_F)^2}$ is adapted to the dominated splitting $E\oplus F$ for the cocycle $DX_t$, where $w=w_E+w_F\in E_x\oplus F_x,
x\in\Gamma, \ \textrm{and} \ \xi$ is an arbitrary positive constant; see \cite[Section 4.1]{arsal2012a}. This means that there exists $\mu>0$ such that
$|DX_t\mid_{E_x}|\cdot|DX_{-t}\mid_{F_{X_t(x)}}|\le e^{-\mu
  t}$ for all $t>0$.

Moreover, from the definition of the inner product and $\wedge$, it follows that

$|\det (DX_t\mid_{F_x})|=\|(\wedge^k DX_t)(u_{1}\wedge ... \wedge u_{k})\|=\|(\wedge^k DX_t)\mid_{\widetilde F}\|\ge e^{\lambda t}$ for all $t>0$, so
$|\cdot|$ is adapted to the volume expanding along $F$.

To conclude, we are left to show that $E$ admits a constant $\omega>0$ such that $|DX_t\mid_{E}|\le e^{-\omega t}$ for all $t>0$.

But since $E$ is uniformly contracted, we know that $X(x)\in F_x$ for all $x\in\Gamma$.

\begin{lemma}
  \label{le:flow-center}
  Let $\Gamma$ be a compact invariant set for a flow
  $X$ of a $C^1$ vector field $X$ on $M$.  Given a
  continuous splitting $T_\Gamma M = E\oplus F$ such
  that $E$ is uniformly contracted, then $X(x)\in F_x$
  for all $x\in \Lambda$.
\end{lemma}

See \cite[Lemma 5.1]{ArArbSal} and \cite[Lemma 3.3]{arsal2012a}.

On the one hand, on each non-singular point $x$ of $\Gamma$ we obtain for each $w\in E_x$
  \begin{align*}
    e^{-\mu t}\ge \frac{|DX_t\cdot w|}{|DX_t\cdot X(x)|}
    =
    \frac{|DX_t\cdot w|}{|X(X_t(x))|}
    \ge
    \frac{|DX_t\cdot w|}{\sup\{|X(z)|:z\in\Gamma\}}
    \ge
    |DX_t\cdot w|,
  \end{align*}
  since we can always choose a small enough constant $\xi>0$
  in such a way that $\sup\{|X(z)|:z\in\Gamma\}\le1$. We
  note that the choice of the positive constant $\xi$ does
  not change any of the previous relations involving
  $|\cdot|$.

  On the other hand, for $\sigma\in\Gamma$ such that
  $X(\sigma)=0$, we fix $t>0$ and, since $\Gamma$
  is a non-trivial invariant set, we can find a
  sequence $x_n\to\sigma$ of regular points of
  $\Gamma$. The continuity of the derivative cocycle
  ensures $|DX_t\mid_{E_\sigma}|=\lim_{n\to\infty}|DX_t\mid_{E_{x_n}}|\le
  e^{-\lambda t}$. Since $t>0$ was arbitrarily chosen,
  we see that $|\cdot|$ is adapted for the contraction
  along $E_\sigma$.

  This shows that $\lambda=\mu$ and completes the proof.
\end{proof}



\def\cprime{$'$}

  \bibliographystyle{abbrv}

\end{document}